\documentclass[reqno]{amsart}
\usepackage{enumerate}
\newtheorem{theorem}{Theorem}[section]
\newtheorem{lemma}[theorem]{Lemma}

\theoremstyle{definition}
\newtheorem{definition}[theorem]{Definition}

\newcommand{\norm}[1]{\left\Vert#1\right\Vert}

\numberwithin{equation}{section}
\usepackage{amsfonts}

\usepackage{amsmath}
\usepackage{amssymb}
\usepackage{cases}
\usepackage{hyperref}
\usepackage{float}
\hypersetup{
   colorlinks,
    citecolor=blue,
    filecolor=black,
    linkcolor=blue,
    urlcolor=magenta
}
\hypersetup{linktocpage}

\begin{document}
\font\nho=cmr10
\def\dive{\operatorname{div}}
\def\cal{\mathcal}
\def\L{\cal L}

\def \ud{\underline }
\def\id{{\indent }}
\def\f{\frac}
\def\non{{\noindent}}
 \def\le{\leqslant} 
 \def\leq{\leqslant} 
\def\rar{\rightarrow}
\def\Rar{\Rightarrow}
\def\ti{\times}
\def\i{\mathbb I}
\def\j{\mathbb J}
\def\si{\sigma}
\def\Ga{\Gamma}
\def\ga{\gamma}
\def\ld{{\lambda}}
\def\Si{\Psi}
\def\f{\mathbf F}
\def\r{\hro{R}}
\def\e{\cal{E}}
\def\B{\cal B}
\def\A{\mathcal{A}}
\def\p{\mathbb P}

\def\tet{\theta}
\def\Tet{\Theta}
\def\hro{\mathbb}
\def\ho{\mathcal}
\def\P{\ho P}
\def\E{\mathcal{E}}
\def\n{\mathbb{N}}
\def\M{\mathbb{M}}
\def\dMu{\mathbf{U}}
\def\dMcs{\mathbf{C}}
\def\dMcu{\mathbf{C^u}}
\def\vk{\vskip 0.2cm}
\def\td{\Leftrightarrow}
\def\df{\frac}
\def\Wei{\mathrm{We}}
\def\Rey{\mathrm{Re}}
\def\s{\mathbb S}
\def\l{\mathcal{L}}
\def\C+{C_+([t_0,\infty))}
\def\o{\cal O}

\title[Stability of the stationary Boussinesq system]{Stability of solutions of stationary Boussinesq systems on weak-Morrey spaces}
\author[T.T. Ngoc]{Tran Thi Ngoc}
\address{Tran Thi Ngoc \hfill\break
Faculty of Fundamental Sciences, East Asia University of Technology, Polyco building, Trinh Van Bo, Nam Tu Liem, Hanoi, Viet
Nam}
\email{ngoctt@eaut.edu.vn or ngoc2tt@gmail.com}
\author[P.T. Xuan]{Pham Truong Xuan}
\address{Pham Truong Xuan \hfill\break
Thang Long Institute of Mathematics and Applied Sciences (TIMAS), Thang Long University \hfill\break
Nghiem Xuan Yem, Hoang Mai, Hanoi, Vietnam} 
\email{xuanpt@thanglong.edu.vn or phamtruongxuan.k5@gmail.com}

\begin{abstract}
In this paper we establish the asymptotic stability of steady solutions for the Boussinesq systems in the framework of Cartesian product of critical weak-Morrey spaces on $\mathbb{R}^n$, where $n \geqslant 3$. In our strategy, we first establish the continuity for the long time of the bilinear terms associated with the mild solutions of the Boussinesq systems, i.e., the bilinear estimates by using only the norm of the present spaces. As a direct consequence, we obtain the existence of global small mild solutions and asymptotic stability of steady solutions of the Boussinesq systems in the class of continuous functions from $[0, \infty)$ to the Cartesian product of critical weak-Morrey spaces. Our techniques consist interpolation of operators, duality, heat semigroup estimates , H\"older and Young inequalities in block spaces (based on Lorentz spaces) that are preduals of Morrey–Lorentz spaces. 
\end{abstract}



\maketitle

\tableofcontents

\section{Introduction}
In this paper, we are concerned about the stability of the steady problem for the viscous Boussinesq system which has the following form
\begin{equation}\label{StBouEq} 
\left\{
  \begin{array}{rll}
 - \Delta \bar{u} + (\bar{u}\cdot\nabla)\bar{u} + \nabla \bar{p} \!\! &= \kappa \bar{\theta} g \quad  &x\in  \mathbb{R}^n, \hfill \\
\nabla \cdot \bar{u} \!\!&=\; 0 \quad &x\in \mathbb{R}^n, \\
 - \Delta \bar{\theta} + (\bar{u} \cdot \nabla)\bar{\theta} \!\!&=\; h \quad &x\in \mathbb{R}^n,\\
(\bar{u},\bar{\theta}) \!\!& \to \; (0,0) \quad & \hbox{as, } |x| \to \infty,\\
\end{array}\right.
\end{equation}
where the unknowns: $\bar{u}(x)$ denotes the velocity of the fluid, $\bar{p}(x)$ is the hydrostatic pressure and $\bar{\theta}$ is the temperature. The given $g(x)$ represents the external force by unit of mass and $h(x)$ is the reference temperature and the positive constant $\kappa$ denotes the coefficient of volume expansion.

In particular, we will determine a new class of steady solutions $(\bar{u},\bar{\theta})$ of \eqref{StBouEq} which is stable for nonsmooth initial disturbance.
Inspirited from previous works \cite{Hi1995,Fe2010,Ko1995}, we describe the stability problem for \eqref{StBouEq} as follows: if the pair $(\bar{u}(x),\bar{\theta}(x))$ is initially perturbed by $(u_0(x),\theta_0(x))$, the the perturbed flow $(\widetilde{u},\widetilde{p},\widetilde{\theta})$ is described by
\begin{equation}\label{BouEq} 
\left\{
  \begin{array}{rll}
 \widetilde{u}_t - \Delta \widetilde{u} + (\widetilde{u} \cdot\nabla) \widetilde{u} + \nabla \widetilde{p} \!\! &= \kappa \widetilde{\theta} g \quad  &x\in  \mathbb{R}^n,\, t>0, \hfill \\
\nabla \cdot \widetilde{u} \!\!&=\; 0 \quad &x\in \mathbb{R}^n,\,  t\geqslant 0, \\
\widetilde{\theta}_t - \Delta \widetilde{\theta} + (\widetilde{u} \cdot \nabla)\widetilde{\theta} \!\!&=\; h \quad &x\in \mathbb{R}^n,\, t>0, \\
\widetilde{u}(x,0) \!\!& = \;\bar{u}_0(x)+u_0(x) \quad & x \in \mathbb{R}^n,\\
\widetilde{\theta}(x,0) \!\!& = \;\bar{\theta}_0(x)+\theta_0(x) \quad & x\in \mathbb{R}^n,\\
(\widetilde{u},\widetilde{\theta}) \!\!& \to \; (0,0) \quad & \hbox{as,  } |x| \to \infty, t>0,\\
\end{array}\right.
\end{equation}
Let $(\widetilde{u},\widetilde{p},\widetilde{\theta})$ be a solution of the problem \eqref{BouEq} and $(u,p,\theta)$ be the disturbance defined by 
$$u(x,t):= \widetilde{u}(x,t)-\bar{u}(x), \, \theta(x,t):=\widetilde{\theta}(x,t)-\bar{\theta}(x),\, p(x,t):=\widetilde{p}(x,t)-\bar{p}(x).$$ 
We have that the triple of functions $(u,p,\theta)$ satisfies the following disturbance system 
\begin{equation}\label{DisturBouEq} 
\left\{
  \begin{array}{rll}
 u_t - \Delta u + (u\cdot\nabla)u +(\bar{u}\cdot\nabla)u + (u\cdot\nabla)\bar{u} + \nabla p \!\! &= \kappa \theta g \quad  &x\in  \mathbb{R}^n,\, t>0, \hfill \\
\nabla \cdot u \!\!&=\; 0 \quad &x\in \mathbb{R}^n,\,  t\geqslant 0, \\
\theta_t - \Delta\theta + (u\cdot \nabla)\theta + (\bar{u}\cdot\nabla)\theta + (u\cdot\nabla)\bar{\theta}   \!\!&=\; 0 \quad &x\in \mathbb{R}^n,\, t>0, \\
u(x,0) \!\!& = \; u_0(x) \quad & x \in \mathbb{R}^n,\\
\theta(x,0) \!\!& = \; \theta_0(x) \quad & x\in \mathbb{R}^n,\\
(u,\theta) \!\!& \to \; (0,0) \quad & \hbox{as,  } |x| \to \infty, t>0,\\
\end{array}\right.
\end{equation}
Therefore, we can study the stability problem for the steady Boussinesq system \eqref{StBouEq} by establishing the existence and large time behavior of global mild solutions to the system \eqref{DisturBouEq}. Note that, if the function $\theta$ vanishes, then the system \eqref{DisturBouEq} becomes the Navier-Stokes equation and our stability resutls can be implied to the ones of the Navier-Stokes equations (for the stability of stationary Navier-Stokes equations see \cite{Ko1995}). Moreover, in the case where the gravitational field satisfies the homogeneous property $g(t,x) = \lambda^2g(\lambda^2 t,\lambda x)$ for $\lambda>0$, the system \eqref{DisturBouEq} presents the following scaling
\begin{equation}\label{invariant}
[u(,x),\theta(t,x)]\to [u_\lambda(t,x),\theta_\lambda(t,x)] = \lambda[u(\lambda^2t,\lambda x),\theta(\lambda^2t,\lambda x)].
\end{equation}

We recall briefly some results about the existence, uniqueness and long-time behavious of the Boussinesq system. The issues of existence and long-time behavior of strong solutions for the Boussinesq system \eqref{BouEq} have attracted the attention of many authors (see for instance \cite{Al2011,Br2019,Ca1980,Fe2006,Fe2008,Fe2010,Ka2008} and references therein). In particular, the existence of solutions in the class $L^p(0,T;L^q(\mathbb{R}^n))$ with suitable $p, q$ and the unbounded function $g$ was established early by Cannon et al. in \cite{Ca1980}. The existence of \eqref{BouEq} in pseudomeasure-type spaces by considering a non-constant and constant gravitational field $g$, respectively was studied by Ferreira et al. \cite{Fe2008} and Karch et al. \cite{Ka2008}. The existence and stability of the global small mild solutions of \eqref{BouEq} in weak-Lorentz spaces were studied by Ferreira et al. in \cite{Fe2006}. These resutls were also extended in Morrey spaces in \cite{Al2011}. More further, the uniqueness of the mild solutions $(u,\theta)$ for \eqref{BouEq} in the subspace with time-weight functions of $C([0,T], L^3(\mathbb{R}^3)\times L^1(\mathbb{R}^3))$ with the initial data $(u_0,\theta_0)\in L^3\times L^1$ on $\mathbb{R}^3$ has been obtained recently by Brandolese et al. in \cite{Br2019}. Recently, the unconditional uniqueness of Boussinesq system has been established by Ferreira and Xuan \cite{Fe2023}.
Besides, the global well-posedness of weak solutions for \eqref{BouEq} with the initial data $(u_0,\theta_0)\in B^{-1}_{\infty,1}\cap L^2\times B^0_{2,1}$ or $(u_0,\theta_0) \in L^2\times B^{-1}_{\infty,1}\cap L^2$ on $\mathbb{R}^2$ were studied in \cite{Ab2007,Da2009} respectively. The existence and time-decays of weak solutions for \eqref{BouEq} with the initial data $(u_0,\theta_0)\in L^2_\sigma\times L^2$ on $\mathbb{R}^3$ were studied by Schonbek et al. in \cite{Br2012}.

Concerning the stability of the stationary solution of the steady problem for the Boussinesq system \eqref{StBouEq}, Hishida established the existence and exponential stability of a global in time strong solution of \eqref{BouEq} near to the steady state in a bounded domain of $\mathbb{R}^3$ in \cite{Hi1995}. After that, Hishida proved also the existence and large time behaviour of a global in time strong solution in an exterior domain of $\mathbb{R}^3$ in \cite{Hi1997} by using $L^{p,\infty}-L^{q,\infty}$-estimates of the semigroup $e^{-tL}$ associated with the corresponding linearized equations of \eqref{BouEq}. Recently, Ferreira et al. \cite{Fe2010} have provided a complete answer to the stability problem of stationary solution of \eqref{StBouEq} in weak-Lorentz spaces $L^{n,\infty}$, where $n \geqslant 3$. They considered the steady solution $(\bar{u},\bar{\theta}) \in L^{n,\infty}_\sigma(\Omega)\times L^{n,\infty}(\Omega)$ (where $\Omega\subset \mathbb{R}^n$) and established the long time behaviour of mild solution of \eqref{DisturBouEq} in $BC((0,\infty);L^{q,\infty}_\sigma(\Omega)\times L^{q,\infty}(\Omega))$, where $q>n$. These results improve the earlier ones of Hishida \cite{Hi1995,Hi1997}. In a related work, Kozono et al. \cite{Ko1995} answered also such problem for the stationary solution of Navier-Stokes equation in the framework of Morrey spaces. On the other hand, the stability of global small mild solution of the Boussinesq systems without the appearance of stationary solutions in the framework of Morrey and weak-Morrey spaces have been also obtained recently in \cite{Al2011} and \cite{VXT2024}, respectively.

In the present paper we consider the stability of the stationary solution $(\bar{u},\bar{\theta})$ of \eqref{StBouEq}, i.e., the existence and large time behaviour of self-similar solutions (whic are invariant by \eqref{invariant}) of the system \eqref{DisturBouEq} in the class of Cartesian product of weak-Morrey spaces $\mathcal{M}^\sigma_{p,\infty,\lambda}\times \mathcal{M}_{p,\infty,\lambda}$ on $\mathbb{R}^n$, where $\mathcal{M}_{p,\infty,\lambda}$ is weak-Morrey space (see Definition \ref{weak-Morrey} below) with $\lambda=n-p$ and $n \geqslant 3$. In particular, in Section \ref{S2} we first recall the notions of Lorentz, Morrey-Lorentz spaces and the construction of predual of Morrey-Lorentz spaces (also called block spaces). Then we recall some estimates for the heat semigroup $\left\{e^{t\Delta}\right\}_{t>0}$ on Morrey-Lorentz and block spaces, where Yamazaki-type estimate is the key to prove the main resutls. In Section \ref{S3} we will prove the bilinear estimates in $H_{p,\infty}$ and $H_{q,r,\infty}$ (see Subsections \ref{Hp} and \ref{Hqr} below) of the bilinear term associated with the mild solution of \eqref{DisturBouEq} and then using these estimates to establish global well-posedness in $BC((0,\infty); \mathcal{M}^\sigma_{p,\infty,\lambda}\times \mathcal{M}_{p,\infty,\lambda})$ and asymptotic stablity in $BC((0,\infty); \mathcal{M}^\sigma_{q,\infty,\lambda}\times \mathcal{M}_{r,\infty,\lambda})$, where $n \geqslant 3$ and $q,r$ satisfy some suitable conditions. Our main resutls are Theorems \ref{Bestimate}, \ref{wellposed}, \ref{BBBestimate} and \ref{Stability}. Our results are generalized the previous ones of the steady Boussinesq systems in weak-$L^p$ spaces obtained by Hishida \cite{Hi1997} and Ferreira et al. \cite{Fe2010} and of the Navier-Stokes equations in Morrey spaces obtained by Kozono et al. \cite{Ko1995}.

\section{Function spaces and properties}\label{S2}
In this section, we recall some definitions and properties about Lorentz, Morrey, Morrey–Lorentz spaces, the preduals of the Morrey-Lorentz spaces and their properties that will be useful for the next sections (in detail see \cite{Fe2016,Fe2023}).
\subsection{Lorentz, Morrey and Morrey-Lorentz spaces}
\begin{definition}
Let $\Omega$ be a subset of $\mathbb{R}^n$, a measurable function $f$ defined on $\Omega$
belongs to the Lorentz space $L^{p,q}(\Omega)$ if the quantity
\begin{subnumcases}{\norm{f}_{L^{p,q}} =}
\left( \int_0^\infty\left[ t^{\frac{1}{p}}f^{**}(t)\right]^q\frac{dt}{t} \right)^{\frac{1}{q}} & $\hbox{if  } 1<p<\infty,\, 1\leq q <\infty,$ \cr
\sup_{t>0} t^{\frac{1}{p}}f^{**}(t) & $\hbox{if  } 1<p\leq \infty,\, q=\infty$\nonumber  
\end{subnumcases}
is finite, where
$$f^{**}(t) = \frac{1}{t}\int_0^tf^{*}(s)ds$$
and
$$f^*(t)=\inf \left\{s>0: m(\left\{ x\in \Omega:|f(x)|>s\right\})\leq t \right\},\, t> 0,$$
with $m$ denoting the Lebesgue measure in $\mathbb{R}^n$.
\end{definition}

The space $L^{p,q}(\Omega)$ with the norm $\norm{f}_{p,q}$ is a Banach space. In particular, $L^p(\Omega)=L^{p,p}(\Omega)$ and $L^{p,\infty}$ is called weak-$L^p$ space or the Marcinkiewicz space.
For $1\leq q_1\leq p \leq q_2\leq \infty$ we have the following relation
\begin{equation}\label{Inclusion}
L^{p,1}(\Omega)\subset L^{p,q_1}(\Omega)\subset L^p(\Omega) \subset L^{p,q_2}(\Omega)\subset L^{p,\infty}(\Omega).
\end{equation}

Let $1<p_1,p_2,p_3\leq \infty$ and $1\leq r_1,r_2,r_3\leq \infty$ be such that $\dfrac{1}{p_3}=\dfrac{1}{p_1}+ \dfrac{1}{p_2}$ and $\dfrac{1}{r_1} + \dfrac{1}{r_2}\geqslant \dfrac{1}{r_3}$. We have the H\"older inequality (see \cite{Hu,One}):
\begin{equation}\label{Holder}
\norm{fg}_{L^{p_3,r_3}} \leq C\norm{f}_{L^{p_1,r_1}}\norm{g}_{L^{p_2,r_2}},
\end{equation}
where $C > 0$ is a constant independent of $f$ and $g$. We also have Young inequality in Lorentz spaces
defined on $\Omega=\mathbb{R}^n$. In particular, if $1<p_1,p_2,p_3\leq \infty$ and $1\leq r_1,r_2,r_3\leq \infty$ be such that $\dfrac{1}{p_3}=\dfrac{1}{p_1}+ \dfrac{1}{p_2}-1$ and $\dfrac{1}{r_1} + \dfrac{1}{r_2}\geqslant \dfrac{1}{r_3}$, then (see \cite{One}):
\begin{equation}
\norm{f*g}_{L^{p_3,r_3}} \leq C\norm{f}_{L^{p_1,r_1}}\norm{g}_{L^{p_2,r_2}}.
\end{equation}
Moreover, in the case $p_1=1$ and $1<p=p_2=p_3\leq \infty$, the following inequality holds (see \cite{BeBr}):
\begin{equation}\label{Young}
\norm{f*g}_{L^{p,\infty}} \leq \frac{p^{1+\frac{1}{p}}}{p-1}\norm{f}_{L^1}\norm{g}_{L^{p,\infty}}.
\end{equation}

We now recall an interpolation property of Lorentz spaces. For $0<p_1<p_2\leq \infty$, $0<\theta<1$, $\dfrac{1}{p} = \dfrac{1-\theta}{p_1}+ \dfrac{\theta}{p_2}$ and $1\leq r_1,r_2,r\leq\infty$, we have (see \cite[Theorems 5.3.1 and 5.3.2]{BeLo}):
\begin{equation}
\left( L^{p_1,r_1},L^{p_2,r_2}\right)_{\theta,r} = L^{p,r},
\end{equation}
where $L^{p_1,r_1}$ is endowed with $\norm{.}^*_{L^{p_1,r_1}}$ instead of $\norm{.}_{L^{p_1,r_1}}$ when $0<p_1\leq r_1$.

By interpolating \eqref{Young}, we get
\begin{equation}\label{InYoung}
\norm{f*g}_{L^{p,r}} \leq C \norm{f}_{L^1}\norm{g}_{L^{p,r}},
\end{equation}
for $1<p\leq \infty$ and $1\leq r \leq \infty$. In the case $1\leq p=r\leq \infty$, the inequality \eqref{InYoung} holds if the space $L^1$ is replaced by
the space of signed finite measures $\mathcal{M}$ endowed with the norm of the total variation 
$\norm{\mu}_{\mathcal{M}} = |\mu|(\mathbb{R}^n)$.

Now we recall the definitions of Morrey and Morrey-Lorentz spaces. 
\begin{definition}
Let $1\leq p,q\leq \infty$ and $0\leq \lambda <n$ be fixed. Moreover, we assume also that $q=\infty$ when $p=\infty$. The Morrey space $\mathcal{M}_{p,\lambda}:=\mathcal{M}_{p,\lambda}(\Omega)$ is defined as the set of all functions $f\in L^p_{loc}(\Omega)$ such that the norm
$$\norm{f}_{p,\lambda} = \sup_{x_0\in \Omega,\rho>0}\rho^{-\frac{\lambda}{p}}\norm{f}_{L^p(D(x_0,\rho)\cap \Omega)}<\infty,$$
where $D(x_0,\rho)=\left\{ x\in \mathbb{R}^n:|x-x_0|<\rho \right\}$. In the case $p=1$, we consider $\mathcal{M}_{1,\lambda}$ as a space of signed measures
with $\norm{.}_{L^1(D(x_0,\rho)\cap \Omega)}$ standing for the total variation of $f$ on $D(x_0,\rho)\cap \Omega$.
\end{definition}
A natural generalization of Morrey space $\mathcal{M}_{p,\lambda}$ is Morrey-Lorentz space.
\begin{definition}\label{weak-Morrey}
We say that $f\in L^{p,q}_{Loc}(\Omega)$ belongs to the Morrey–Lorentz space $\mathcal{M}_{p,q,\lambda}:=\mathcal{M}_{p,q,\lambda}(\Omega)$ when the quantity
\begin{equation}\label{NormLM}
\norm{f}_{p,q,\lambda}:=\norm{f}_{\mathcal{M}_{p,q,\lambda}} = \sup_{x_0\in \Omega, \rho>0} \rho^{-\frac{\lambda}{p}}\norm{f}_{L^{p,q}(D(x_0,\rho)\cap \Omega)}<\infty.
\end{equation}
In the case $p=q$, we have $\mathcal{M}_{p,p,\lambda}=\mathcal{M}_{p,\lambda}$. In the case $q=\infty$, we obtain the weak-Morrey space $\mathcal{M}_{p,\infty,\lambda}$.
\end{definition}
If $p=q=1$ or $1<p\leq \infty$ and $1\leq q \leq \infty$, the space $\mathcal{M}_{p,q,\lambda}$ endowed with the norm $\norm{.}_{p,q,\lambda}$ is a Banach space. If $p=1$ and $1<q\leq \infty$, we consider \eqref{NormLM} with $\norm{f}^*_{L^{p,q}(D(x_0,\rho)\cap \Omega)}$ instead if $\norm{f}_{L^{p,q}(D(x_0,\rho)\cap \Omega)}$ and $\mathcal{M}_{p,q,\lambda}$ is a complete quasi-normed space.

We have the continuous inclusions
$$L^n(\mathbb{R}^n)\hookrightarrow L^{n,\infty}(\mathbb{R}^n)\hookrightarrow \mathcal{M}_{p,n-p}(\mathbb{R}^n)\hookrightarrow \mathcal{M}_{p,\infty,n-p}(\mathbb{R}^n).$$
Moreover, in view of \eqref{Inclusion} and \eqref{Holder}, we have the following inclusion
\begin{equation}
\mathcal{M}_{p_2,q_2,\lambda} \subset \mathcal{M}_{p_1,q_1,\mu},
\end{equation}
where $1\leq p_1\leq p_2\leq \infty$, $\dfrac{n-\mu}{p_1} = \dfrac{n-\lambda}{p_2}$ and $1\leq q_2\leq q_1\leq \infty$.

Setting $\tau_{p,\lambda}:=\dfrac{n-\lambda}{p}$. By a simple change of variables, one obtains the scaling property
\begin{equation}
\norm{f(cx)}_{\mathcal{M}_{p,q,\lambda}(\Omega)} = c^{-\tau_{p,\lambda}}\norm{f(x)}_{\mathcal{M}_{p,q,\lambda}(c\Omega)}, \hbox{  for all  } c>0.
\end{equation}

We have the H\"older inequality on the framework of Morrey-Lorentz spaces as follows (see \cite[Lemma 2.1]{Fe2016}):
\begin{lemma}
For $1<p_0,p_1,r\leq \infty$ and $0\leq \beta, \lambda_0,\lambda_1<n$ satisfy that $\dfrac{1}{r}=\dfrac{1}{p_0}+\dfrac{1}{p_1}$, $\dfrac{\beta}{r} = \dfrac{\lambda_0}{p_0}+\dfrac{\lambda_1}{p_1}$ and $s\geqslant 1$ satisfies that $\dfrac{1}{q_0} + \dfrac{1}{q_1}\geqslant \dfrac{1}{s}$, the following inequality holds
\begin{equation}\label{HolderWM}
\norm{fg}_{r,s,\beta} \leq C\norm{f}_{p_0,q_0,\lambda_0}\norm{g}_{p_1,q_1,\lambda_1},
\end{equation}
where $C>0$ is a universal constant. In the case $r=s=1$ and $\dfrac{1}{q_0}+\dfrac{1}{q_1}=1$, then \eqref{HolderWM} also holds by considering $\mathcal{M}_{1,1,\beta}$ as $\mathcal{M}_{1,\beta}$ (similarly when either $p_0=1$ or $p_1=1$). 
\end{lemma}

In order to prove the estimate of bilinear term associated with the mild solution of the Boussinesq system i.e.,''bilinear estimate'' we need the following dispersive and smoothing estimates of the heat semigroup $\left\{ e^{t\Delta} \right\}_{t\geqslant 0}$ on Morrey-Lorentz spaces (see \cite[Lemma 2.2]{Fe2016}):
\begin{lemma}\label{Disp}
Let $\Omega=\mathbb{R}^n$, $m\in \left\{ 0\right\} \cup \mathbb{N}$, $1<p,r\leq \infty$, $1\leq q\leq d \leq \infty$, $0\leq \lambda,\mu <\infty$ and $\tau_{r,\mu}= \dfrac{n-\mu}{r} \leq \tau_{p,\lambda}= \dfrac{n-\lambda}{p}$. Assume further that $\lambda=\mu$ when $p\leq r$. Then, there exists a constant $C>0$ satisfying that
\begin{equation}\label{disp}
\norm{\nabla_x^me^{t\Delta}\varphi}_{r,d,\mu} \leq Ct^{-\frac{m}{2}-\frac{1}{2}(\tau_{p,\lambda}-\tau_{r,\mu})}\norm{\varphi}_{p,q,\lambda}
\end{equation}
for all $\varphi \in \mathcal{M}_{p,q,\lambda}$. The estimate \eqref{disp} is also valid when either $r=d=1$ or $1\leq p=q\leq \infty$ or $1<r<\infty$ and $p=q=1$.
\end{lemma}

\subsection{Block spaces and properties}
We recall the predual of Morrey-Lorentz spaces $\mathcal{M}_{p,q,\lambda}$ and their properties.
We begin with the definition of block spaces (see \cite{Bl,Fe2016,Zo}).
Let $1 < p \leq \infty$, $1 \leq q \leq \infty$ and $\kappa \geqslant 0$. We assume also that $q = \infty$ when $p = \infty$.
Recall that $D(a, \rho)$ stands for the open ball in $\mathbb{R}^n$ with center a and radius $\rho>0$. A measurable function $b(x)$ is a $(p,q,\kappa)$-block if $\mathrm{supp}(b) \subset D(a, \rho)$ for some $a \in \mathbb{R}^n$ and $\rho > 0$, and satisfies
\begin{equation}
\rho^{\frac{\kappa}{p}}\norm{b}_{L^{p,q}(D(a,\rho))} \leqslant 1.
\end{equation}
The block space $\mathcal{P}\mathcal{D}_{p,q,\kappa}$ is the set of all measurable functions $h(x)$ that can be written as
\begin{equation}
h(x) = \sum_{k=1}^\infty\alpha_kb_k(x), \hbox{  for a.e.,  } x\in \mathbb{R}^n,
\end{equation}
where $b_k(x)$ is a $(p,q,\kappa)$-block and $\sum_{k=1}^\infty\alpha_k<\infty$. The space $\mathcal{P}\mathcal{D}_{p,q,\kappa}$ is a Banach space endowed with the norm
\begin{equation}
\norm{h}_{\mathcal{P}\mathcal{D}_{p,q,\kappa}} = \inf \left\{ \sum_{k=1}^\infty|\alpha_k|<\infty; \, h= \sum_{k=1}^\infty\alpha_kb_k \hbox{  where $b_k$'s are $(p,q,\kappa)$-blocks}  \right\}.
\end{equation}
For $1<p_2\leq p_1\leq \infty$ and $\omega_{p_2,\kappa_2} = \dfrac{n+\kappa_2}{p_2} = \dfrac{n+\kappa_1}{p_1} = \omega_{p_1,\kappa_1}$ we have the following inclusion
\begin{equation}
\mathcal{P}\mathcal{D}_{p_1,q_1,\kappa_1}\subset \mathcal{P}\mathcal{D}_{p_2,q_2,\kappa_2}.
\end{equation}

The relation between the predual of Morrey-Lorentz space and block space is given in the following lemma (see \cite[Lemma 3.1]{Fe2016}):
\begin{lemma}\label{DualBlock}
Let $1<p\leq \infty$, $1<q\leq \infty$, $0\leq \lambda<n$, $\kappa>0$, $\dfrac{\lambda}{p}=\dfrac{\kappa}{p'}$, $\dfrac{1}{p}+\dfrac{1}{p'}=1$ and $\dfrac{1}{q}+\dfrac{1}{q'}=1$. We assume further that $q=\infty$ when $p=\infty$. Then, we have
\begin{equation}
\left( \mathcal{P}\mathcal{D}_{p',q',\kappa} \right)^* = \mathcal{M}_{p,q,\lambda}.
\end{equation}
\end{lemma}
In order to prove the bilinear estimate in the next section we need to use the linear and continuous of the interpolation of operators in block spaces (see \cite[Lemma 3.2]{Fe2016}).
\begin{lemma}\label{Interpolation}
Let $1<p,p_1,p_0\leq \infty$, $1<q,q_1,q_0\leq \infty$ and $\kappa,\kappa_1,\kappa_0\geqslant 0$ be such that $\dfrac{1}{p} = \dfrac{1-\theta}{p_0}+\dfrac{\theta}{p_1}$ and $\dfrac{\kappa}{p} = \dfrac{(1-\theta)\kappa_0}{p_0} + \dfrac{\theta \kappa_1}{p_1}$ with $\theta\in (0,\,1)$. If $X_0$ and $X_1$ are Banach spaces and $\mathcal{T}$ is a continuous linear operator satisfying that
\begin{equation}
\mathcal{T}: \mathcal{P}\mathcal{D}_{p_0,q_0,\kappa_0}\to X_0 \hbox{  and  } \mathcal{T}: \mathcal{P}\mathcal{D}_{p_1,q_1,\kappa_1}\to X_1,
\end{equation}
Then, $\mathcal{T}: \mathcal{P}\mathcal{D}_{p,q,\kappa}\to (X_0,X_1)_{\theta,q}$ is linear and continuous.
\end{lemma}
As a direct consequence of the above lemma we have the continuous inclusion
\begin{equation}\label{IncluBlock}
\mathcal{P}\mathcal{D}_{p,q,\kappa} \hookrightarrow \left( \mathcal{P}\mathcal{D}_{p_0,q_0,\kappa_0}, \mathcal{P}\mathcal{D}_{p_1,q_1,\kappa_1} \right)_{\theta,q},
\end{equation}
where $\dfrac{1}{p} = \dfrac{1-\theta}{p_0}+\dfrac{\theta}{p_1}$ and $\dfrac{\kappa}{p} = \dfrac{(1-\theta)\kappa_0}{p_0}+\dfrac{\theta\kappa}{p_1}$. The reverse inclusion of \eqref{IncluBlock} does not hold true due to the family of Morrey spaces (more generally, the
one of Morrey–Lorentz spaces) is not closed under the functor $(.,.)_{\theta,q}$.

We have the H\"older-type inequality in block space (see \cite[Lemma 3.5]{Fe2016}):
\begin{lemma}\label{HolderBlockSpace}
Let $1<p,p_1,p_0\leq \infty$, $1<q,q_1,q_0\leq \infty$ and $\kappa,\kappa_1,\kappa_0\geqslant 0$ be such that $\dfrac{1}{p} = \dfrac{1}{p_0}+\dfrac{1}{p_1}$ and $\dfrac{1}{q} \leq \dfrac{1}{1_0} + \dfrac{1}{q_1}$. Then, we have
\begin{equation}\label{HolderBlock}
\norm{fg}_{\mathcal{P}\mathcal{D}_{p,q,\kappa}}\leq C \norm{f}_{\mathcal{P}\mathcal{D}_{p_0,q_0,\kappa_0}}\norm{g}_{\mathcal{P}\mathcal{D}_{p_1,q_1,\kappa_1}},
\end{equation}
where $C > 0$ is a constant independent of $f$ and $g$. Moreover, if $p=q=1$ and $\dfrac{1}{q_0}+\dfrac{1}{q_1}\geqslant 1$, then \eqref{HolderBlock} is also valid. 
\end{lemma}

Finally, we recall estimates for the heat semigroup $\left\{ e^{t\Delta} \right\}_{t\geqslant 0}$ acting on $\mathcal{P}\mathcal{D}_{p,q,\kappa}$-spaces (see \cite[Lemma 3.6]{Fe2016}):
\begin{lemma}\label{disBlock}
Let $1<p_1,p_2\leq \infty$, $1\leq q_1\leq q_2\leq \infty$, $\kappa_1,\kappa_2\geqslant 0$, $m\in \left\{ 0\right\} \cup \mathbb{M}$ and $\omega_{p_2,\kappa_2}\leq \omega_{p_1,\kappa_1}$, where $\omega_{p_1,\kappa_i} = \dfrac{n+\kappa_i}{p_i}$. Moreover, we assume that $0\leq \kappa_1\dfrac{p'_1}{p_1} = \kappa_2\dfrac{p'_2}{p_2} <n $ when $p_1\leq p_2$, where $\dfrac{1}{p_1}+\dfrac{1}{p_1'} = \dfrac{1}{p_2}+\dfrac{1}{p_2'}=1$. Suppose also that $q_i = \infty$ when
$p_i = \infty$. Then, there is a constant $C > 0$ such that
\begin{equation}
\norm{\nabla_x^me^{t\Delta}\varphi}_{\mathcal{P}\mathcal{D}_{p_2,q_2,\kappa_2}} \leq Ct^{-\frac{m}{2}-\frac{1}{2}(\omega_{p_1,\kappa_1} - \omega_{p_2,\kappa_2})}\norm{\varphi}_{\mathcal{P}\mathcal{D}_{p_1,q_1,\kappa_1}}
\end{equation}
for all $\varphi \in \mathcal{P}\mathcal{D}_{p_1,q_1,\kappa_1}$.
\end{lemma}
Using Lemmas \ref{Interpolation} and \ref{disBlock}, we can establish Yamazaki-type estimate (see the origin Yamazaki estimate in \cite[pages 648-650]{Ya2000}) in the framework of $\mathcal{P}\mathcal{D}_{p,q,\kappa}$-spaces (see the proof in \cite[Lemma 5.1]{Fe2016}).
\begin{lemma}\label{YamazakiType}
Let $1<p<r<\infty$ and $\kappa,\alpha>0$ be such that $0\leq \alpha \dfrac{r'}{r} = \kappa\dfrac{p'}{p}<n$, where $\dfrac{1}{r}+\dfrac{1}{r'}=1$ and $\dfrac{1}{p}+\dfrac{1}{p'}=1$. Let $\omega_{r,\alpha} = \dfrac{n+\alpha}{r}$ and $\omega_{p,\kappa}=\dfrac{n+\kappa}{p}$. There is a constant $C>0$ satisfying that
\begin{equation}
\int_0^\infty s^{\frac{1}{2}(\omega_{p,\kappa}-\omega_{r,\alpha})-\frac{1}{2}}\norm{\nabla_xe^{s\Delta}\varphi}_{\mathcal{P}\mathcal{D}_{r,1,\alpha}} ds \leq C\norm{\varphi}_{\mathcal{P}\mathcal{D}_{p,1,\kappa}},
\end{equation}
for all $\varphi \in \mathcal{P}\mathcal{D}_{p,1,\kappa}$.
\end{lemma}

\section{Asymptotic stability of steady Boussinesq system}\label{S3}
\subsection{The disturbance system}
Since $\dive{\bar u} = \dive u =0$ we have the following identities
$$(u\cdot \nabla)\theta = \dive (\theta u),\, (\bar{u}\cdot \nabla) u =\dive{(u\otimes \bar{u})}, (u\cdot \nabla) \bar{u} =\dive{(\bar{u}\otimes u)} \hbox{   and   } (u\cdot \nabla)u = \dive(u\otimes u).$$
Therefore, the Boussinesq system \eqref{DisturBouEq} can be rewritten as
\begin{equation}\label{BouEq1} 
\left\{
  \begin{array}{rll}
 u_t - \Delta u + \mathbb{P}\dive(u\otimes u+u\otimes \bar{u}+\bar{u}\otimes u) \!\! &= \kappa\mathbb{P}(\theta g) \quad  & x\in \mathbb{R}^n,\, t>0, \hfill \\
\dive u \!\!&=\; 0 \quad &x\in \mathbb{R}^n,\, t\geq 0, \\
\theta_t - \Delta \theta + \dive(\theta u + \theta\bar{u} + \bar{\theta}u) \!\!&=\; 0 \quad &x\in \mathbb{R}^n,\, t>0, \\
u(x,0) \!\!& = \;u_0(x) \quad & x\in \mathbb{R}^n,\\
\theta(x,0) \!\!& = \;\theta_0(x) \quad & x\in \mathbb{R}^n,\\
(u,\theta) \!\!& \to \;(0,0) \quad & \hbox{as  } |x| \to \infty,\, t>0,
\end{array}\right.
\end{equation}
where $\mathbb{P}$ is the Leray projector which is a matrix $n\times n$ with elements $(\mathbb{P})_{k,j} = \delta_{kj} + \mathcal{R}_k\mathcal{R}_j$, where $\mathcal{R}_j = \partial_j(-\Delta)^{-\frac{1}{2}}\, (j=1,2...n)$ are the Reisz transform. We notice that the Reisz transform $\mathcal{R}_j$ is continuous form $\mathcal{M}_{p,q,\lambda}(\mathbb{R}^n)$ to itself, for each $j=1,2...n$ (see \cite[Lemma 2.3]{Fe2016}).

\def\A{\mathcal{A}}
We denote the set of all vector $u:\mathbb{R}^n\to \mathbb{R}$ such that $\dive u = 0$ and $u$ belongs to $\mathcal{M}_{p,q,\lambda}$ by $\mathcal{M}_{p,q,\lambda}^\sigma$.
We will consider the well-posedness of \eqref{BouEq1} in the following time-dependent functional space based on the Cartesian product of weak-Morrey spaces 
$$H_{p,\infty}:= BC(\mathbb{R}_+,\mathcal{M}^\sigma_{p,\infty,\lambda}(\mathbb{R}^n)\times \mathcal{M}_{p,\infty,\lambda}(\mathbb{R}^n)), \hbox{  where  } \lambda=n-p,$$ 
endowed with the norm
$$\norm{ \begin{bmatrix} u\\ 
\theta
\end{bmatrix} }_{H_{p,\infty}}: = \sup_{t>0} \norm{u(\cdot,t)}_{p,\infty,\lambda} + \norm{\theta(\cdot,t)}_{p,\infty,\lambda}.$$
To this purpose, we set the operator
$L:=\begin{bmatrix}-\Delta& 0\\ 
0& -\Delta
\end{bmatrix}$ acting on the Cartesian product space $\mathcal{M}^\sigma_{p,\infty,\lambda}(\mathbb{R}^n)\times \mathcal{M}_{p,\infty,\lambda}(\mathbb{R}^n)$, $\lambda=n-p$. Then, equation \eqref{BouEq1} is rewritten {in the operator matrix form} as 
\begin{align}\label{AbstractEquation}
\begin{cases}
\displaystyle\frac{\partial }{\partial t} \begin{bmatrix}
u\\
\theta
\end{bmatrix}
  + L\begin{bmatrix}
u\\
\theta
\end{bmatrix} = \begin{bmatrix}
\mathbb{P}[- \dive (u\otimes u)] \\
\dive(-\theta u)
\end{bmatrix} + \begin{bmatrix}
\mathbb{P}[- \dive (u\otimes \bar{u}+ \bar{u}\otimes u)] \\
- \dive(\theta \bar{u} + \bar{\theta} u )
\end{bmatrix} + \begin{bmatrix}
\kappa\mathbb{P}(\theta g)\\
0
\end{bmatrix}\\
\begin{bmatrix}
u(0)\\
\theta(0)
\end{bmatrix}  = \begin{bmatrix}
u_0\cr \theta_0
\end{bmatrix}.
\end{cases}
\end{align}

According to Duhamel's principle we define the mild solution of the equation \eqref{AbstractEquation} as follows. 
\begin{definition}   
Assume that the steady solution of system \eqref{StBouEq} satisfies that $\begin{bmatrix}
\bar{u}\\
\bar{\theta} 
\end{bmatrix} \in \mathcal{M}^\sigma_{p,\infty,\lambda}\times \mathcal{M}_{p,\infty,\lambda}$, $\lambda=n-p$.
The mild solution of the equation \eqref{AbstractEquation} is defined as the solution of the integral equation 
\begin{equation}\label{mildsol}
\begin{bmatrix}
u\\
\theta 
\end{bmatrix}(t) = e^{-tL}\begin{bmatrix}
u_0\\
\theta_0
\end{bmatrix} + B\left(\begin{bmatrix}
u\\
\theta
\end{bmatrix}, \begin{bmatrix}
u\\
\theta
\end{bmatrix} \right)(t) + \mathcal{T}\left(\begin{bmatrix}
u\\
\theta
\end{bmatrix} \right)(t) + T_g(\theta)(t),
\end{equation}
where bilinear and linear-coupling operators used in the above equation are
\begin{equation}\label{Bilinear}
B\left(\begin{bmatrix}
u\\
\theta
\end{bmatrix}, \begin{bmatrix}
v\\
\xi
\end{bmatrix} \right)(t): = -\int_0^t \nabla_x e^{-(t-s)L}\begin{bmatrix}\mathbb{P}(u\otimes v)\\
u\xi
\end{bmatrix}(s)ds,
\end{equation}
\begin{equation}
\mathcal{T}\left(\begin{bmatrix}
u\\
\theta
\end{bmatrix} \right)(t): = -\int_0^t \nabla_x e^{-(t-s)L} \begin{bmatrix}
\mathbb{P}(u\otimes \bar{u}+ \bar{u}\otimes u) \\
\theta \bar{u} + \bar{\theta} u 
\end{bmatrix}(s)ds
\end{equation}
and
\begin{equation}\label{Couple}
T_g(\theta)(t):= \int_0^t e^{-(t-s)L}\begin{bmatrix}\kappa\mathbb{P}(\theta g)\\
0
\end{bmatrix}(s) ds.
\end{equation}
\end{definition}

\subsection{Global well-posedness in $H_{p,\infty}$}\label{Hp}
We consider the linear operator
\begin{equation}\label{linearOp}
\mathcal{C}\begin{bmatrix}
f_1\\
f_2
\end{bmatrix}(x) = \int_0^\infty\nabla_x e^{-sL}\begin{bmatrix}
f_1\\
f_2
\end{bmatrix}(\cdot,s)ds
\end{equation}
By using the duality and Lemma \ref{YamazakiType} we estimate \eqref{linearOp} in the following lemma.
\begin{lemma}\label{LinearEst}
Let $n\geqslant 3$, $1<r<l<\infty$ and $0\leq \lambda <n$. Assume that $\tau_{r,\lambda}-\tau_{l,\lambda}=1$, where $\tau_{l,\lambda}=\dfrac{n-\lambda}{l}$ and $\tau_{r,\lambda}=\dfrac{n-\lambda}{r}$.
There is a constant $C > 0$ such that
\begin{equation}\label{linearEs}
\norm{\mathcal{C}\begin{bmatrix}
f_1\\
f_2
\end{bmatrix}}_{l,\infty,\lambda} \leq C\sup_{t>0}\norm{\begin{bmatrix}
f_1\\
f_2
\end{bmatrix}(\cdot,t)}_{r,\infty,\lambda}
\end{equation}
for all $(f_1,f_2)\in BC(\mathbb{R}_+, \mathcal{M}_{r,\infty,\lambda}\times \mathcal{M}_{r,\infty,\lambda})$ and we denote $\norm{.}_{\mathcal{M}_{l,\infty,\lambda}\times \mathcal{M}_{l,\infty,\lambda}}:=\norm{.}_{l,\infty,\lambda}$ and $\norm{.}_{\mathcal{M}_{r,\infty,\lambda}\times \mathcal{M}_{r,\infty,\lambda}}:=\norm{.}_{r,\infty,\lambda}$. The supremum on $t > 0$ in \eqref{linearEs} is the essential one.
\end{lemma}
\begin{proof}
The proof has been given in \cite[Lemma 3.2]{Fe2023}.
\end{proof}
Using inequality \eqref{linearEs} we state and prove the bilinear estimate in the following theorm.
\begin{theorem}\label{Bestimate}(Bilinear estimate in $H_{p,\infty}$).
Let $n\geqslant 3$ and $2<p\leq n$. Let $B(\cdot,\cdot)$ be the bilinear form \eqref{Bilinear}. There exists a constant $K > 0$ such that
\begin{equation}\label{bestimate}
\norm{B\left( \begin{bmatrix}
u\\
\theta
\end{bmatrix}, \begin{bmatrix}
v\\
\xi
\end{bmatrix} \right)}_{H_{p,\infty}} \leq K \norm{\begin{bmatrix}
u\\
\theta
\end{bmatrix}}_{H_{p,\infty}}\norm{\begin{bmatrix}
v\\
\xi
\end{bmatrix}}_{H_{p,\infty}}
\end{equation} 
for all $\begin{bmatrix}
u\\
\theta
\end{bmatrix}, \begin{bmatrix}
v\\
\xi
\end{bmatrix} \in H_{p,\infty}$.
\end{theorem}
\begin{proof}
We rewrite the bilinear term \eqref{Bilinear} as follows
$$B\left(\begin{bmatrix}
u\\
\theta
\end{bmatrix}, \begin{bmatrix}
v\\
\xi
\end{bmatrix} \right)(t): = -\int_0^t \nabla_x e^{-(t-s)L}\begin{bmatrix}\mathbb{P}(u\otimes v)\\
u\xi
\end{bmatrix}(s)ds = \mathcal{C}\begin{bmatrix}
f_{1t}\\
f_{2t}
\end{bmatrix}$$
with $f_{1t}$ and $f_{2t}$ given by
$$f_{1t} = \mathbb{P}(u\otimes v)(\cdot, t-s), \, f_{2t} = (u\xi)(\cdot,t-s) \hbox{  for a.e.  } s\in (0,t),$$
$$f_{1t} = f_{2t} = 0 \hbox{  for } s\in (t,\infty).$$

Let $2<p\leq n$ and $r=\dfrac{p}{2}$. It follows from H\"older's inequality \eqref{HolderWM} and the continuous of the Leray projector $\mathbb{P}$ in $\mathcal{M}_{r,\infty,\lambda}$ (see \cite[Lemma 2.3]{Fe2016}) that
\begin{eqnarray}
\sup_{0<s<t}\norm{\begin{bmatrix}
f_{1t}(\cdot,s)\\
f_{2t}(\cdot,s)
\end{bmatrix}}_{r,\infty,n-p} &=& \sup_{0<s<t}\norm{\begin{bmatrix}
\mathbb{P}(u\otimes v)(\cdot, t-s)\\
(u\xi)(\cdot,t-s)
\end{bmatrix}}_{r,\infty,n-p}\cr
&\leq& \sup_{0<s<t}\norm{\begin{bmatrix}
u(\cdot, s)\\
\varphi(\cdot,s)
\end{bmatrix}}_{p,\infty,n-p} \sup_{0<s<t}\norm{\begin{bmatrix}
v(\cdot, s)\\
\xi(\cdot,s)
\end{bmatrix}}_{p,\infty,n-p}.
\end{eqnarray}
Applying Lemma \ref{LinearEst} with $l=p$, $r=\dfrac{p}{2}$, $\lambda=n-p$, $\tau_{l,\lambda}=\dfrac{n-\lambda}{l}=1$ and $\tau_{r,\lambda}=\dfrac{n-\lambda}{r}=2$ we obtain that
\begin{eqnarray*}
\sup_{t>0} \norm{B\left(\begin{bmatrix}
u\\
\theta
\end{bmatrix}, \begin{bmatrix}
v\\
\xi
\end{bmatrix} \right)}_{p,\infty,n-p} &=& \sup_{t>0} \norm{\mathcal{C}\begin{bmatrix}
f_{1t}\\
f_{2t}
\end{bmatrix}}_{p,\infty,n-p}\cr
&\leq& C\sup_{t>0}\sup_{0<s<t}\norm{\begin{bmatrix}
f_{1t}(\cdot,s)\\
f_{2t}(\cdot,s)
\end{bmatrix}}_{r,\infty,n-p}\cr
&\leq& K \sup_{t>0}\norm{\begin{bmatrix}
u(\cdot, t)\\
\varphi(\cdot,t)
\end{bmatrix}}_{p,\infty,n-p}\sup_{t>0}\norm{\begin{bmatrix}
v(\cdot, t)\\
\xi(\cdot,t)
\end{bmatrix}}_{p,\infty,n-p}.
\end{eqnarray*}
The proof is completed.
\end{proof}
We recall the classical lemma in a generic Banach space, its proof is based on the Banach fixed point theorem (in details see \cite{Ca1996,Fe2006}).
\begin{lemma}\label{FixedPoint}
Let $X$ be a Banach space with norm $\norm{\cdot}_X$, $T:X\to X$ be a linear continuous map with norm $\norm{T}\leq \tau<1$ and $B:X\times X \to X$ be a continuous bilinear map, that is, there is a constant $K>0$ such that for all $x_1, x_2\in X$ we have
$$\norm{B(x_1,x_2)}_X \leq K \norm{x_1}_X\norm{x_2}_X.$$
Then, if $0<\varepsilon < \dfrac{(1-\tau)^2}{4K}$ and for any vector $y\in X,\, y\neq 0$ such that $\norm{y}_X<\varepsilon$, there is a solution $x\in X$ for the equation $x=y + B(x,x) + T(x)$ such that 
$$\norm{x}_X \leq \dfrac{2\varepsilon}{1-\tau}.$$
The solution $x$ is unique in the ball $\bar{B}\left( 0,\dfrac{2\varepsilon}{1-\tau} \right)$. Moreover, the solution depends continuously on $y$ in the following sense: if $\norm{widetilde{y}}_X \leq \varepsilon, \, \widetilde{x} = \widetilde{y} + B(\widetilde{x},\widetilde{x}) + T(\widetilde{x})$ and $\norm{\widetilde{x}}_X < \dfrac{2\varepsilon}{1-\tau}$, then
$$\norm{x-\widetilde{x}}_X \leq \frac{1-\tau}{(1-\tau)^2-4K\varepsilon}\norm{y-\widetilde{y}}_X.$$
\end{lemma}
Using the bilinear estimate obtained in Theorem \ref{Bestimate} and Lemma \ref{FixedPoint} we state and prove the existence of global small solution for \eqref{AbstractEquation} in the following theorem.
\begin{theorem}\label{wellposed}(Gloabl-in-time well-posedness of \eqref{AbstractEquation} in $H_{p,\infty}$). Let $n\geqslant 3$, $2<p\leq n$ and $p\leq \min \left\{ n,\, 2b\right\}$. We suppose that $\begin{bmatrix}
\bar{u}\\
\bar{\theta}
\end{bmatrix} \in \mathcal{M}^\sigma_{p,\infty,\lambda}\times \mathcal{M}_{p,\infty,\lambda}$, where $\lambda=n-p$. If $\norm{\begin{bmatrix}
u_0\\
\theta_0
\end{bmatrix}}_{p,\infty,\lambda}$, $\norm{\begin{bmatrix}
\bar{u}\\
\bar{\theta}
\end{bmatrix}}_{p,\infty,\lambda}$ and $\sup\limits_{t>0}t^{1-\frac{p}{2b}}\norm{g(\cdot,t)}_{b,\infty,\lambda}$ are small enough, then there exists a unique solution $\begin{bmatrix}
\hat u\\
\hat\theta
\end{bmatrix}$ of the Cauchy problem \eqref{AbstractEquation} in a small ball of $H_{p,\infty}$.
Here, we denote $\norm{\begin{bmatrix}
v\\
\xi
\end{bmatrix}}_{p,\infty,\lambda}:=\norm{\begin{bmatrix}
v\\
\xi
\end{bmatrix}}_{\mathcal{M}^\sigma_{p,\infty,\lambda}\times \mathcal{M}_{p,\infty,\lambda}}$.
\end{theorem}
\begin{proof}
First, by using \eqref{disp} with $\dfrac{1}{d} = \dfrac{1}{p}+\dfrac{1}{b}$ (where $\dfrac{p}{2b}<1$), H\"older's inequality \eqref{HolderWM} and the continuous of Leray projector $\mathbb{P}$ we estimate the linear-coupling part as follows
\begin{eqnarray}\label{TEst}
\norm{T_g(\theta)}_{p,\infty,\lambda}&\leq& \int_0^t \norm{e^{-(t-s)L}\begin{bmatrix}\kappa\mathbb{P}(\theta g)\\
0
\end{bmatrix}(s)}_{p,\infty,\lambda} ds\cr
&\leq& \kappa\int_0^t (t-s)^{-\frac{1}{2}\left( \frac{p}{d} -1 \right)}\norm{\begin{bmatrix}\theta g\\
0
\end{bmatrix}(s)}_{d,\infty,\lambda} ds\cr
&\leq& \kappa\int_0^t (t-s)^{-\frac{p}{2b}}\norm{g(s)}_{b,\infty,\lambda}\norm{\theta(s)}_{p,\infty,n-p} ds\cr
&\leq& \kappa \sup_{t>0}t^{1-\frac{p}{2b}}\norm{g(t)}_{b,\infty,\lambda}\sup_{t>0}\norm{\theta}_{p,\infty,n-p} \int_0^t (t-s)^{-\frac{p}{2b}} s^{-1+\frac{p}{2b}} ds\cr
&\leq& \kappa \sup_{t>0}t^{1-\frac{p}{2b}}\norm{g(t)}_{b,\infty,\lambda}\sup_{t>0}\norm{\theta}_{p,\infty,n-p} \int_0^1 (1-z)^{-\frac{p}{2b}} z^{-1+\frac{p}{2b}} dz \cr
&\leq& \kappa \sup_{t>0}t^{1-\frac{p}{2b}}\norm{g(t)}_{b,\infty,\lambda}\sup_{t>0}\norm{\theta}_{p,\infty,n-p} \cr
&&\times \left( \int_0^{1/2} (1-z)^{-\frac{p}{2b}} z^{-1+\frac{p}{2b}} dz + \int_{1/2}^1 (1-z)^{-\frac{p}{2b}} z^{-1+\frac{p}{2b}} dz \right) \cr
&\leq& \kappa \sup_{t>0}t^{1-\frac{p}{2b}}\norm{g(t)}_{b,\infty,\lambda}\sup_{t>0}\norm{\theta}_{p,\infty,n-p} \left( \frac{2^{1+\frac{p}{b}}b}{p} + \frac{1}{1-\frac{p}{2b}} \right)\cr
&\leq& \kappa M\sup_{t>0}t^{1-\frac{p}{2b}}\norm{g(t)}_{b,\infty,\lambda} \norm{\begin{bmatrix}
u\\
\theta
\end{bmatrix}}_{H_{p,\infty}},
\end{eqnarray}
where $M:=\dfrac{2^{1+\frac{p}{b}}b}{p} + \dfrac{1}{1-\frac{p}{2b}}$. 
Moreover, by using Theorem \ref{Bestimate} we have that 
\begin{equation}\label{Testimate}
\norm{\mathcal{T} \left(\begin{bmatrix}
u\\
\theta
\end{bmatrix} \right)}_{H_{p,\infty}} \leq 2K \norm{\begin{bmatrix}
\bar{u}\\
\bar{\theta}
\end{bmatrix}}_{H_{p,\infty}} \norm{\begin{bmatrix}
u\\
\theta
\end{bmatrix}}_{H_{p,\infty}} 
\end{equation}
We set the linear operator
$$\mathbb{T}\left(\begin{bmatrix}
u\\
\theta
\end{bmatrix} \right) = \mathcal{T}\left(\begin{bmatrix}
u\\
\theta
\end{bmatrix} \right) + T_g(\theta).$$
Using \eqref{TEst} and \eqref{Testimate} we have that
\begin{equation}\label{T}
\norm{\mathbb{T}\left(\begin{bmatrix}
u\\
\theta
\end{bmatrix} \right)}_{H_{p,\infty}} \leq \left( 2K\begin{bmatrix}
\bar{u}\\
\bar{\theta}
\end{bmatrix}_{H_{p,\infty}} + \kappa M\sup_{t>0}t^{1-\frac{p}{2b}}\norm{g(t)}_{b,\infty,\lambda} \right)\norm{\begin{bmatrix}
u\\
\theta
\end{bmatrix}}_{H_{p,\infty}}.
\end{equation}

Using \eqref{T}, Theorem \ref{Bestimate} and Lemma \ref{FixedPoint} for equation \eqref{AbstractEquation} which is equivalent with 
\begin{equation*}
\begin{bmatrix}
u\\
\theta
\end{bmatrix}(t)= e^{-(t-s)L}\begin{bmatrix}
u_0\\
\theta_0
\end{bmatrix} + B\left( \begin{bmatrix}
u\\
\theta
\end{bmatrix}, \begin{bmatrix}
u\\
\theta
\end{bmatrix} \right) (t) + \mathbb{T}\left( \begin{bmatrix}
u\\
\theta
\end{bmatrix} \right)(t),
\end{equation*}
we have that: there are constants $C_n, K_1, \delta, \eta$ such that if 
$$\norm{\begin{bmatrix}
u_0\\
\theta_0
\end{bmatrix}}_{H_{p,\infty}} < \delta \hbox{  and  } 2K\begin{bmatrix}
\bar{u}\\
\bar{\theta}
\end{bmatrix}_{H_{p,\infty}} + \kappa M\sup_{t>0}t^{1-\frac{p}{2b}}\norm{g(t)}_{b,\infty,\lambda}< \eta,$$
then \eqref{AbstractEquation} has a global mild solution $\begin{bmatrix}
u\\
\theta
\end{bmatrix} \in H_{p,\infty}$ satisfying that
$$\norm{\begin{bmatrix}
u\\
\theta
\end{bmatrix}}_{H_{p,\infty}} \leq \frac{2C_n\delta}{1-K_1\eta}.$$
\end{proof}

\subsection{Global well-posedness in $H_{q,r,\infty}$}\label{Hqr}
We consider the following Cartesian product of weak-Morrey space
$$\mathcal{M}^\sigma_{q,\infty,\mu}(\mathbb{R}^n)\times \mathcal{M}_{r,\infty,\nu}(\mathbb{R}^n)$$
endowned with the norm
$$\norm{\begin{bmatrix} u\\ 
\theta
\end{bmatrix}}_{q,\infty,\mu;r,\infty,\nu}: = \norm{u}_{q,\infty,\mu} + \norm{\theta}_{r,\infty,\nu}.$$
In order to establish the asymptotic stability in the next section
we need to prove the existence of global small solution of \eqref{AbstractEquation}
in the following time-dependent functional space
$$H_{q,r,\infty}: = \left\{ \begin{bmatrix}
u\cr \theta
\end{bmatrix} \in H_{p,\infty} : \sup_{t>0}\norm{\begin{bmatrix}
t^{\alpha/2}u\cr t^{\beta/2}\theta
\end{bmatrix}}_{q,\infty,\lambda;r,\infty,\lambda} <\infty  \right\},\hbox{  where  } \lambda=n-p,$$
endowed with the norm
$$\norm{ \begin{bmatrix} u\\ 
\theta
\end{bmatrix} }_{H_{q,r,\infty}}: = \norm{\begin{bmatrix}
u\cr \theta
\end{bmatrix}}_{H_{p,\infty}} + \sup_{t>0}\norm{\begin{bmatrix}
t^{\alpha/2}u\cr t^{\beta/2}\theta
\end{bmatrix}}_{q,\infty,\lambda;r,\infty,\lambda},$$
where $\alpha = 1-\dfrac{p}{q}$ and $\beta = 1-\dfrac{p}{r}$.
 
We can extend the bilinear estimate obtained in Theorem \ref{Bestimate} to $H_{q,r,\infty}$ as follows:
\begin{theorem}\label{BBBestimate}(Bilinear estimate in $H_{q,r,\infty}$)
Let $n\geqslant 3$ and $1<p<q<r<\infty$. Let $B(\cdot,\cdot)$ be the bilinear form \eqref{Bilinear}.
\begin{enumerate}
\item[(i)] There exists a constant $\widetilde{K} > 0$ such that
\begin{equation}\label{BBestimate}
\norm{B\left( \begin{bmatrix}
u\\
\theta
\end{bmatrix}, \begin{bmatrix}
v\\
\xi
\end{bmatrix} \right)}_{H_{q,r,\infty}} \leq \widetilde{K} \norm{\begin{bmatrix}
u\\
\theta
\end{bmatrix}}_{H_{q,r,\infty}}\norm{\begin{bmatrix}
v\\
\xi
\end{bmatrix}}_{H_{q,r,\infty}}
\end{equation} 
for all $\begin{bmatrix}
u\\
\theta
\end{bmatrix}, \begin{bmatrix}
v\\
\xi
\end{bmatrix} \in H_{q,r,\infty}$.

\item[(ii)] There exists a constant $\widehat{K} > 0$ such that
\begin{equation}\label{BBBBestimate}
\norm{B\left( \begin{bmatrix}
u\\
\theta
\end{bmatrix}, \begin{bmatrix}
v\\
\xi
\end{bmatrix} \right)}_{H_{q,r,\infty}} \leq \widehat{K} \norm{\begin{bmatrix}
u\\
\theta
\end{bmatrix}}_{H_{p,\infty}}\norm{\begin{bmatrix}
v\\
\xi
\end{bmatrix}}_{H_{q,r,\infty}}
\end{equation} 
for all $\begin{bmatrix}
u\\
\theta
\end{bmatrix} \in H_{p,\infty}, \begin{bmatrix}
v\\
\xi
\end{bmatrix} \in H_{q,r,\infty}$.

\end{enumerate}

\end{theorem}
\begin{proof}
We have that
\begin{eqnarray}
B\left(\begin{bmatrix}
u\\
\theta
\end{bmatrix}, \begin{bmatrix}
v\\
\xi
\end{bmatrix} \right)(t): &=& -\int_0^t \nabla_x e^{-(t-s)L}\begin{bmatrix}\mathbb{P}(u\otimes v)\\
u\xi
\end{bmatrix}(s)ds \cr
&=& -\int_0^t \begin{bmatrix}
\nabla_x e^{(t-s)\Delta}\mathbb{P}(u\otimes v)\\
\nabla_x e^{(t-s)\Delta}(u \xi)
\end{bmatrix}(s)ds\cr
&=& \begin{bmatrix}
B_1(u,v)\\
B_2(u,\xi),
\end{bmatrix}
\end{eqnarray}
where
\begin{equation}
B_1(u,v):= -\int_0^t \nabla_x e^{(t-s)\Delta}\mathbb{P}(u\otimes v)(s)ds,\,\, B_2(u,\xi):= \int_0^t \nabla_xe^{(t-s)\Delta}(u \xi)(s)ds.
\end{equation}

\underline{Assertion $(i)$.}
Let $\dfrac{1}{d}=\dfrac{1}{q}+\dfrac{1}{r}$. Setting $\lambda=n-p$, by using Lemma  \ref{Disp} and the continuous of Leray projector $\mathbb{P}$ we estimate that
\begin{eqnarray}\label{ine1}
&&\norm{\begin{bmatrix}
t^{\frac{\alpha}{2}}B_1(u,v)\\
t^{\frac{\beta}{2}}B_2(u,\xi),
\end{bmatrix}}_{q,\infty, \lambda; r,\infty, \lambda} \leq C\int_0^t \norm{\begin{bmatrix}
t^{\frac{\alpha}{2}}(t-s)^{-\frac{1}{2}-\frac{p}{2q}}(u\otimes v)\\
t^{\frac{\beta}{2}}(t-s)^{-\frac{1}{2}-\frac{p}{2q}}(u \xi)
\end{bmatrix}(s)}_{\frac{q}{2},\infty,\lambda;d,\infty,\lambda}  ds\cr
&\leq& C\int_0^t \left( t^{\frac{\alpha}{2}}(t-s)^{-\frac{1}{2}-\frac{p}{2q}}\norm{(u\otimes v)(s)}_{\frac{q}{2},\infty,\lambda} +  t^{\frac{\beta}{2}}(t-s)^{-\frac{1}{2}-\frac{p}{2q}}\norm{(u \xi)(s)}_{d,\infty,\lambda} \right) ds\cr
&:=& C(J_1 + J_2),
\end{eqnarray}
where
\begin{eqnarray*}
&&J_1= \int_0^t t^{\frac{\alpha}{2}}(t-s)^{-\frac{1}{2}-\frac{p}{2q}}\norm{(u\otimes v)(s)}_{\frac{q}{2},\infty,\lambda} ds\cr
&&J_2= \int_0^t t^{\frac{\beta}{2}}(t-s)^{-\frac{1}{2}-\frac{p}{2q}}\norm{(u \xi)(s)}_{d,\infty,\lambda} ds.
\end{eqnarray*}

By using H\"older inequality in Lemma \ref{HolderWM} we can estimate $J_1$ and $J_2$ as follows
\begin{eqnarray}\label{ine2}
J_1 &\leq&  \int_0^t t^{\frac{\alpha}{2}}(t-s)^{-\frac{1}{2}-\frac{p}{2q}}\norm{u(s)}_{q,\infty,\lambda}\norm{v(s)}_{q,\infty,\lambda} ds\cr
&\leq& t^{\frac{\alpha}{2}}\int_0^t(t-s)^{-\frac{1}{2}-\frac{p}{2q}}s^{-\alpha} ds \sup_{t>0}t^{\frac{\alpha}{2}}\norm{u(t)}_{q,\infty,\lambda}\sup_{t>0}t^{\frac{\alpha}{2}}\norm{v(t)}_{q,\infty,\lambda} \cr
&\leq& t^{\frac{1}{2}-\frac{p}{2q}-\frac{\alpha}{2}}\int_0^1 (1-s)^{-\frac{1}{2}-\frac{p}{2q}}s^{-\alpha} ds \sup_{t>0}t^{\frac{\alpha}{2}}\norm{u(t)}_{q,\infty,\lambda}\sup_{t>0}t^{\frac{\alpha}{2}}\norm{v(t)}_{q,\infty,\lambda}\cr
&\leq& \int_0^1 (1-s)^{-\frac{1}{2}-\frac{p}{2q}}s^{-\alpha} ds \sup_{t>0}t^{\frac{\alpha}{2}}\norm{u(t)}_{q,\infty,\lambda}\sup_{t>0}t^{\frac{\alpha}{2}}\norm{v(t)}_{q,\infty,\lambda} \cr
&\leq& \left( \int_0^{1/2} (1-s)^{-\frac{1}{2}-\frac{p}{2q}}s^{-\alpha} ds + \int_{1/2}^1 (1-s)^{-\frac{1}{2}-\frac{p}{2q}}s^{-\alpha} ds  \right)\cr
&&\times \sup_{t>0} \norm{\begin{bmatrix}
t^{\frac{\alpha}{2}}u(t)\\
t^{\frac{\beta}{2}}\theta(t),
\end{bmatrix}}_{q,\infty,\lambda;r,\infty,\lambda} 
\sup_{t>0} \norm{\begin{bmatrix}
t^{\frac{\alpha}{2}}v(t)\\
t^{\frac{\beta}{2}}\xi(t),
\end{bmatrix}}_{q,\infty,\lambda;r,\infty,\lambda}\cr
&\leq& \left( 2^{\frac{1}{2}+\frac{p}{2q}}\int_0^{1/2}s^{-\alpha} ds + 2^{1-\frac{p}{q}}\int_{1/2}^1 (1-s)^{-\frac{1}{2}-\frac{p}{2q}} ds  \right)\cr
&&\times \sup_{t>0} \norm{\begin{bmatrix}
t^{\frac{\alpha}{2}}u(t)\\
t^{\frac{\beta}{2}}\theta(t),
\end{bmatrix}}_{q,\infty,\lambda;r,\infty,\lambda} 
\sup_{t>0} \norm{\begin{bmatrix}
t^{\frac{\alpha}{2}}v(t)\\
t^{\frac{\beta}{2}}\xi(t),
\end{bmatrix}}_{q,\infty ,\lambda;r,\infty ,\lambda}\cr
&\leq& \left( \frac{q 2^{\frac{1}{2}-\frac{p}{2q}}}{p} + \frac{2^{\frac{1}{2}-\frac{p}{2q}}}{\frac{1}{2}-\frac{p}{2q}} \right)\sup_{t>0} \norm{\begin{bmatrix}
t^{\frac{\alpha}{2}}u(t)\\
t^{\frac{\beta}{2}}\theta(t),
\end{bmatrix}}_{q,\infty ,\lambda;r,\infty ,\lambda} 
\sup_{t>0} \norm{\begin{bmatrix}
t^{\frac{\alpha}{2}}v(t)\\
t^{\frac{\beta}{2}}\xi(t),
\end{bmatrix}}_{q,\infty ,\lambda;r,\infty ,\lambda}\cr
&\leq& C_1 \sup_{t>0} \norm{\begin{bmatrix}
t^{\frac{\alpha}{2}}u(t)\\
t^{\frac{\beta}{2}}\theta(t),
\end{bmatrix}}_{q,\infty ,\lambda;r,\infty ,\lambda} 
\sup_{t>0} \norm{\begin{bmatrix}
t^{\frac{\alpha}{2}}v(t)\\
t^{\frac{\beta}{2}}\xi(t),
\end{bmatrix}}_{q,\infty ,\lambda;r,\infty ,\lambda},
\end{eqnarray}
where
$$C_1 = \frac{q 2^{\frac{1}{2}-\frac{p}{2q}}}{p} + \frac{2^{\frac{1}{2}-\frac{p}{2q}}}{\frac{1}{2}-\frac{p}{2q}}.$$
By the same way we can estimate that
\begin{eqnarray}\label{ine3}
J_2 &\leq& \int_0^1 (1-s)^{-\frac{1}{2}-\frac{p}{2q}}s^{-\frac{\alpha+\beta}{2}} ds \sup_{t>0}t^{\frac{\alpha}{2}}\norm{u(t)}_{q,\infty,\lambda}\sup_{t>0}t^{\frac{\beta}{2}}\norm{\xi(t)}_{r,\infty,\lambda} \cr 
&\leq& C_2 \sup_{t>0} \norm{\begin{bmatrix}
t^{\frac{\alpha}{2}}u(t)\\
t^{\frac{\beta}{2}}\theta(t),
\end{bmatrix}}_{q,\infty ,\lambda;r,\infty ,\lambda} 
\sup_{t>0} \norm{\begin{bmatrix}
t^{\frac{\alpha}{2}}v(t)\\
t^{\frac{\beta}{2}}\xi(t),
\end{bmatrix}}_{q,\infty ,\lambda;r,\infty ,\lambda},
\end{eqnarray}
where
$$C_2 = \frac{2^{\frac{1}{2}-\frac{p}{2q}}}{\frac{p}{2q}+\frac{p}{2r}} + \frac{2^{\frac{1}{2}-\frac{p}{2r}}}{\frac{1}{2}-\frac{p}{2q}}.$$
Combining \eqref{ine1}, \eqref{ine2}, \eqref{ine3} and the bilinear estimate \eqref{bestimate} in $H_{p,\infty}$ we obtain the bilinear estimate \eqref{BBestimate} in $H_{q,r,\infty}$.

\underline{Assertion $(ii)$.} 
By applying the linear estimate \eqref{linearEs} in Lemma  \ref{LinearEst} for $r=\dfrac{pq}{p+q}, \, l=q, \, \lambda=n-p$ and the continuous of Leray projector $\mathbb{P}$, we have 
\begin{eqnarray}\label{1}
\norm{t^{\frac{\alpha}{2}}B_1(u,v)}_{q,\infty,\lambda} &=& t^{\frac{\alpha}{2}}\norm{\int_0^t \nabla_x e^{(t-s)\Delta}\begin{bmatrix}\mathbb{P}(u\otimes v)\\0\end{bmatrix}(s)ds}_{q,\infty,\lambda}\cr
&\leq& \widehat{C}\sup_{t>0}\left(t^{\frac{\alpha}{2}}\norm{\begin{bmatrix}(u\otimes v)\\0\end{bmatrix}(\cdot,t)}_{\frac{pq}{p+q},\infty,\lambda} \right)\cr
&\leq& \widehat{C}\sup_{t>0}\left(t^{\frac{\alpha}{2}}\norm{u(\cdot,t)}_{p,\infty,\lambda}\norm{v(\cdot,t)}_{q,\infty,\lambda}\right)\cr
&\leq& \widehat{C}\sup_{t>0}\norm{\begin{bmatrix}u\\\theta\end{bmatrix}(\cdot,t)}_{p,\infty,\lambda}\sup_{t>0} \norm{\begin{bmatrix}t^{\frac{\alpha}{2}}v\\t^{\frac{\beta}{2}}\xi\end{bmatrix}(\cdot,t)}_{q,\infty,\lambda}\cr
&\leq& \widehat{C}\norm{\begin{bmatrix}u\\\theta\end{bmatrix}}_{H_{p,\infty}}\norm{\begin{bmatrix}v\\\xi\end{bmatrix}}_{H_{q,r,\infty}}
\end{eqnarray}
Applying again the linear estimate \eqref{linearEs} in Lemma \ref{LinearEst} for $r=\dfrac{pq}{p+q}, \, l=q, \, \lambda=n-p$, we have also that
\begin{eqnarray}\label{2}
\norm{t^{\frac{\beta}{2}}B_2(v,\xi)}_{r,\infty,\lambda} &=& t^{\frac{\beta}{2}}\norm{\int_0^t \begin{bmatrix}0\\\nabla_x e^{(t-s)\Delta}(u\xi)\end{bmatrix}}_{r,\infty,\lambda}\cr
&\leq& \widehat{C}\sup_{t>0}\left( t^{\frac{\beta}{2}}\norm{\begin{bmatrix}0\\(u\xi)\end{bmatrix}(\cdot,t)}_{\frac{pr}{p+r},\infty,\lambda}\right)\cr
&\leq& \widehat{C}\sup_{t>0}\left(t^{\frac{\beta}{2}}\norm{u(\cdot,t)}_{p,\infty,\lambda}\norm{\xi(\cdot,t)}_{r,\infty,\lambda}\right)\cr
&\leq& \widehat{C}\sup_{t>0}\norm{\begin{bmatrix}u\\\theta\end{bmatrix}(\cdot,t)}_{p,\infty,\lambda}\sup_{t>0}\left( \norm{\begin{bmatrix}t^{\frac{\alpha}{2}}v\\t^{\frac{\beta}{2}}\xi\end{bmatrix}(\cdot,t)}_{r,\infty,\lambda}\right)\cr
&\leq& \widehat{C}\norm{\begin{bmatrix}u\\\theta\end{bmatrix}}_{H_{p,\infty}} \norm{\begin{bmatrix}v\\\xi\end{bmatrix}}_{H_{q,r,\infty}}.
\end{eqnarray}
Combining estimates \eqref{1}, \eqref{2}  and the bilinear estimate in $H_{p,\infty}$ obtained in Theorem \ref{Bestimate}, we obtain the bilinear estimate \eqref{BBBBestimate}.
\end{proof}

Now we state and prove the existence of global small mild solution of \eqref{AbstractEquation} in $H_{q,r,\infty}$.
\begin{theorem}\label{WellposedHqr}(Global-in-time well-posedness of \eqref{AbstractEquation} in $H_{q,r,\infty}$) Let $n\geqslant 3$, $2<p<q<r <\infty$, $r>q':=\dfrac{q}{q-1}$ and $p\leq \min \left\{ n,\, 2b\right\}$, $\dfrac{1}{p}<\dfrac{1}{b}+\dfrac{1}{r}<\min\left\{ \dfrac{2}{p} + \dfrac{1}{q},\, 1\right\}$. We suppose that $\begin{bmatrix}
\bar{u}\\
\bar{\theta}
\end{bmatrix} \in \mathcal{M}^\sigma_{p,\infty,\lambda}\times \mathcal{M}_{p,\infty,\lambda}$, where $\lambda=n-p$. If $\norm{\begin{bmatrix}
u_0\\
\theta_0
\end{bmatrix}}_{p,\infty,\lambda}$, $\norm{\begin{bmatrix}
\bar{u}\\
\bar{\theta}
\end{bmatrix}}_{H_{p,\infty}}$ and $\sup\limits_{t>0}t^{1-\frac{p}{2b}}\norm{g(\cdot,t)}_{b,\infty,\lambda}$ are small enough, then there exists a unique solution $\begin{bmatrix}
u\\
\theta
\end{bmatrix}$ of the Cauchy problem \eqref{AbstractEquation} in a small ball of $H_{q,r,\infty}$.
\end{theorem}
\begin{proof}
We have that \eqref{AbstractEquation} is equivalent to
\begin{eqnarray}
\begin{bmatrix}
t^{\frac{\alpha}{2}}u\\ 
t^{\frac{\beta}{2}}\theta 
\end{bmatrix}(t) &=& e^{-tL}\begin{bmatrix}
t^{\frac{\alpha}{2}}u_0\\
t^{\frac{\beta}{2}}\theta
\end{bmatrix} - \int_0^t \nabla_x e^{-(t-s)L} \begin{bmatrix}
t^{\frac{\alpha}{2}}\mathbb{P}(u\otimes u)\\
t^{\frac{\beta}{2}}u\theta
\end{bmatrix}(s) ds \cr
&& - \int_0^t \nabla_x e^{-(t-s)L}\begin{bmatrix}
t^{\frac{\alpha}{2}}\mathbb{P}(u\otimes \bar{u} + \bar{u}\otimes u)\\
t^{\frac{\beta}{2}}(\theta\bar{u}+ u\bar{\theta})
\end{bmatrix}(s) ds + \int_0^t e^{-(t-s)L}\begin{bmatrix}\kappa e^{t^{\frac{\alpha}{2}}}\mathbb{P}(\theta g)\\.
0
\end{bmatrix}(s) ds.
\end{eqnarray}
We have that
\begin{eqnarray}\label{IIne3}
&&\norm{\int_0^t e^{-(t-s)L}\begin{bmatrix}\kappa e^{t^{\frac{\alpha}{2}}}\mathbb{P}(\theta g)\\.
0
\end{bmatrix}(s) ds}_{q,\infty,\lambda;r,\infty,\lambda} \cr
&\leq& t^{\frac{\alpha}{2}}\int_0^t (t-s)^{1-\frac{p}{2b}+\frac{\beta}{2}-\frac{\alpha}{2}-1}s^{-1+\frac{p}{2b}-\frac{\beta}{2}}s^{\frac{\beta}{2}}\norm{\theta(s)}_{r,\infty,\lambda}ds \sup_{t>0}t^{1-\frac{p}{2b}}\norm{g(t)}_{b,\infty,\lambda}\cr
&\leq& \int_0^1 (t-s)^{-\frac{p}{2b}+\frac{\beta-\alpha}{2}}s^{-1+\frac{p}{2b}-\frac{\beta}{2}}(ts)^{\frac{\beta}{2}}\norm{\theta(ts)}_{r,\infty,\lambda}ds \sup_{t>0}t^{1-\frac{p}{2b}}\norm{g(t)}_{b,\infty,\lambda}\cr
&\leq& L \kappa\sup_{t>0}t^{1-\frac{p}{2b}}\norm{g(t)}_{b,\infty,\lambda} \norm{\begin{bmatrix}
u\\
\theta
\end{bmatrix}}_{H_{q,r,\infty}},
\end{eqnarray}
where
$$L = \int_0^1 (t-s)^{-\frac{p}{2b}+\frac{\beta-\alpha}{2}}s^{-1+\frac{p}{2b}-\frac{\beta}{2}}(ts)^{\frac{\beta}{2}} ds <\infty$$
since the fact that $2<p<q<r <\infty$, $r>q':=\dfrac{q}{q-1}$ and $p\leq \min \left\{ n,\, 2b\right\}$, $\dfrac{1}{p}<\dfrac{1}{b}+\dfrac{1}{r}<\min\left\{ \dfrac{2}{p} + \dfrac{1}{q},\, 1\right\}$.
Moreover, applying assertion $(ii)$ in Theorem \ref{BBBestimate} we have that
\begin{equation}\label{IIne2}
\norm{\int_0^t \nabla_x e^{-(t-s)L}\begin{bmatrix}
t^{\frac{\alpha}{2}}\mathbb{P}(u\otimes \bar{u} + \bar{u}\otimes u)\\
t^{\frac{\beta}{2}}(\theta\bar{u}+ u\bar{\theta})
\end{bmatrix}(s) ds}_{q,\infty,\lambda;r,\infty,\lambda} \leq 2\widehat{K}\norm{\begin{bmatrix}
\bar{u}\\
\bar{\theta}
\end{bmatrix}}_{H_{p,\infty}} \norm{\begin{bmatrix}
u\\
\theta
\end{bmatrix}}_{H_{q,r,\infty}}.
\end{equation}
Combining \eqref{IIne3} and \eqref{IIne2} we obtain that
\begin{equation}\label{IIne1}
\norm{\mathbb{T}\left(\begin{bmatrix}
u\\
\theta
\end{bmatrix}\right)}_{H_{q,r,\infty}} \leq \left( 2\widehat{K}\norm{\begin{bmatrix}
\bar{u}\\
\bar{\theta}
\end{bmatrix}}_{H_{p,\infty}}  + L\kappa \sup_{t>0}t^{1-\frac{p}{2b}}\norm{g(t)}_{b,\infty,\lambda} \right)\norm{\begin{bmatrix}
u\\
\theta
\end{bmatrix}}_{H_{q,r,\infty}}.
\end{equation}
By using estimate \eqref{disp}, the initial data can be estimated as follows
\begin{equation}
\norm{e^{-tL}\begin{bmatrix}
\hat{u}_0\\
\hat{\theta}_0
\end{bmatrix}}_{q,\infty,\lambda;r,\infty,\lambda} \leq C\norm{\begin{bmatrix}
u_0\\
\theta_0
\end{bmatrix}}_{p,\infty,\lambda}. 
\end{equation}
Hence
\begin{equation}\label{IIne0}
\norm{e^{-tL}\begin{bmatrix}
\hat{u}_0\\
\hat{\theta}_0
\end{bmatrix}}_{H_{q,r,\infty}} \leq C\norm{\begin{bmatrix}
u_0\\
\theta_0
\end{bmatrix}}_{p,\infty,\lambda}.
\end{equation}
Using \eqref{IIne1}, \eqref{IIne0}, Theorem \ref{BBBestimate} and Lemma \ref{FixedPoint} we have that: 
there are constants $C_{q,r}, K_2, \delta_{q,r}, \eta_{q,r}$ such that if 
$$\norm{\begin{bmatrix}
u_0\\
\theta_0
\end{bmatrix}}_{p,\infty,\lambda} < \delta_{q,r} \hbox{  and  } 2\widehat{K}\norm{\begin{bmatrix}
\bar{u}\\
\bar{\theta}
\end{bmatrix}}_{p,\infty,\lambda} + L\kappa\sup_{t>0}t^{1-\frac{p}{2b}}\norm{g(t)}_{b,\infty,\lambda}< \eta_{q,r},$$
then \eqref{AbstractEquation} has a global mild solution $\begin{bmatrix}
u\\
\theta
\end{bmatrix} \in H_{q,r,\infty}$ satisfying that
$$\norm{\begin{bmatrix}
u\\
\theta
\end{bmatrix}}_{H_{q,r,\infty}} \leq \frac{2C_{q,r}\delta_{q,r}}{1-K_2\eta_{q,r}}.$$
\end{proof}

\subsection{Asymptotic stablity}
In this subsection, by using bilinear estimates obtained in Theorem \ref{Bestimate} and Theorem \ref{BBBestimate}, we prove the large time behaviour of global mild solutions of system \eqref{DisturBouEq}. This asymptotic stability shows the stability of the stationary Boussinesq system \eqref{StBouEq}.
\begin{theorem}\label{Stability}(Asymptotic stability)
Let $n\geqslant 3$, $2<p<q<r <\infty$, $r>q':=\dfrac{q}{q-1}$ and $p\leq \min \left\{ n,\, 2b\right\}$, $\dfrac{1}{p}<\dfrac{1}{b}+\dfrac{1}{r}<\min\left\{ \dfrac{2}{p} + \dfrac{1}{q},\, 1\right\}$.
Assume that $\begin{bmatrix}
u\\
\theta 
\end{bmatrix}$ is a mild solution of \eqref{AbstractEquation} given by Theorem \ref{WellposedHqr} corresponding to the steady solution $\begin{bmatrix}
\bar{u}\\
\bar{\theta} 
\end{bmatrix} \in \mathcal{M}^\sigma_{p,\infty,\lambda}\times \mathcal{M}_{p,\infty,\lambda}, \, \lambda=n-p$ and initial data $\begin{bmatrix}
u_0\\
\theta_0 
\end{bmatrix} \in \mathcal{M}^\sigma_{p,\infty,\lambda}\times \mathcal{M}_{p,\infty,\lambda}$.
If $\lim\limits_{t\to\infty} \norm{e^{-tL}\begin{bmatrix}
t^{\frac{\alpha}{2}}u_0\\
t^{\frac{\beta}{2}}\theta_0
\end{bmatrix}}_{q,\infty,\lambda;r,\infty,\lambda} =0$, then
\begin{equation}
\lim\limits_{t\to\infty} \norm{\begin{bmatrix}
t^{\frac{\alpha}{2}}u\\
t^{\frac{\beta}{2}}\theta
\end{bmatrix}(\cdot,t)}_{q,\infty,\lambda;r,\infty,\lambda} =0.
\end{equation} 
\end{theorem}
\begin{proof}
We recall that
\begin{eqnarray}
\begin{bmatrix}
t^{\frac{\alpha}{2}}u\\ 
t^{\frac{\beta}{2}}\theta 
\end{bmatrix}(\cdot,t) &=& e^{-tL}\begin{bmatrix}
t^{\frac{\alpha}{2}}u_0\\
t^{\frac{\beta}{2}}\theta
\end{bmatrix} - \int_0^t \nabla_x e^{-(t-s)L} \begin{bmatrix}
t^{\frac{\alpha}{2}}\mathbb{P}(u\otimes u)\\
t^{\frac{\beta}{2}}u\theta
\end{bmatrix}(s) ds \cr
&& - \int_0^t \nabla_x e^{-(t-s)L}\begin{bmatrix}
t^{\frac{\alpha}{2}}\mathbb{P}(u\otimes \bar{u} + \bar{u}\otimes u)\\
t^{\frac{\beta}{2}}(\theta\bar{u}+ u\bar{\theta})
\end{bmatrix}(s) ds + \int_0^t e^{-(t-s)L}\begin{bmatrix}\kappa e^{t^{\frac{\alpha}{2}}}\mathbb{P}(\theta g)\\.
0
\end{bmatrix}(s) ds.
\end{eqnarray}
Hence
\begin{eqnarray}
\norm{\begin{bmatrix}
t^{\frac{\alpha}{2}}u\\ 
t^{\frac{\beta}{2}}\theta 
\end{bmatrix}(\cdot,t)}_{q,\infty,\lambda;r,\infty,\lambda} &\leq& \norm{e^{-tL}\begin{bmatrix}
t^{\frac{\alpha}{2}}u_0\\
t^{\frac{\beta}{2}}\theta
\end{bmatrix}}_{q,\infty,\lambda;r,\infty,\lambda} + \norm{\int_0^t \nabla_x e^{-(t-s)L} \begin{bmatrix}
t^{\frac{\alpha}{2}}\mathbb{P}(u\otimes u)\\
t^{\frac{\beta}{2}}u\theta
\end{bmatrix}(s) ds}_{q,\infty,\lambda;r,\infty,\lambda} \cr
&& + \norm{\int_0^t \nabla_x e^{-(t-s)L}\begin{bmatrix}
t^{\frac{\alpha}{2}}\mathbb{P}(u\otimes \bar{u} + \bar{u}\otimes u)\\
t^{\frac{\beta}{2}}(\theta\bar{u}+ u\bar{\theta})
\end{bmatrix}(s) ds}_{q,\infty,\lambda;r,\infty,\lambda} \cr
&&+ \norm{\int_0^t e^{-(t-s)L}\begin{bmatrix}\kappa e^{t^{\frac{\alpha}{2}}}\mathbb{P}(\theta g)\\.
0
\end{bmatrix}(s) ds}_{q,\infty,\lambda;r,\infty,\lambda}.
\end{eqnarray}
Using Theorem \ref{BBBestimate} and estimate \eqref{IIne3} we obtain the
\begin{eqnarray}\label{limsup}
\limsup_{t\to\infty}\norm{\begin{bmatrix}
t^{\frac{\alpha}{2}}u\\ 
t^{\frac{\beta}{2}}\theta 
\end{bmatrix}(\cdot,t)}_{q,\infty,\lambda;r,\infty,\lambda} &\leq& \limsup_{t\to\infty}\norm{e^{-tL}\begin{bmatrix}
t^{\frac{\alpha}{2}}u_0\\
t^{\frac{\beta}{2}}\theta_0
\end{bmatrix}}_{q,\infty,\lambda;r,\infty,\lambda} + \limsup_{t\to\infty}\widetilde{K}\norm{ \begin{bmatrix}
t^{\frac{\alpha}{2}}u\\
t^{\frac{\beta}{2}}\theta
\end{bmatrix}}^2_{q,\infty,\lambda;r,\infty,\lambda} \cr
&& + 2\limsup_{t\to\infty}\widehat{K}\norm{\begin{bmatrix}\bar{u}\\
\bar{\theta} \end{bmatrix}}_{p,\infty,\lambda}\norm{\begin{bmatrix}t^{\frac{\alpha}{2}}u\\
t^{\frac{\beta}{2}}\theta \end{bmatrix}}_{q,\infty,\lambda;r,\infty,\lambda} \cr
&&+ L\kappa \limsup_{t\to\infty} t^{1-\frac{p}{2b}}\norm{g(t)}_{b,\infty,\lambda}\norm{\begin{bmatrix}t^{\frac{\alpha}{2}}u\\
t^{\frac{\beta}{2}}\theta \end{bmatrix}}_{q,\infty,\lambda;r,\infty,\lambda}\cr
&\leq& \left( \frac{2\widetilde{K}C_{qr}\delta_{qr}}{1-K_2\eta_{qr}} + 2\widehat{K}\norm{\begin{bmatrix}\bar{u}\\
\bar{\theta} \end{bmatrix}}_{p,\infty,\lambda} + L\kappa \limsup_{t\to\infty} t^{1-\frac{p}{2b}}\norm{g(t)}_{b,\infty,\lambda} \right)\cr
&&\times \limsup_{t\to\infty}\norm{\begin{bmatrix}
t^{\frac{\alpha}{2}}u\\
t^{\frac{\beta}{2}}\theta
\end{bmatrix}}_{q,\infty,\lambda;r,\infty,\lambda}\cr
&\leq& \left( \frac{2\widetilde{K}C_{qr}\delta_{qr}}{1-K_2\eta_{qr}} + \eta_{q,r} \right) \limsup_{t\to\infty}\norm{\begin{bmatrix}
t^{\frac{\alpha}{2}}u\\
t^{\frac{\beta}{2}}\theta
\end{bmatrix}}_{q,\infty,\lambda;r,\infty,\lambda},
\end{eqnarray}
where the constants $K_2,C_{qr},\delta_{qr},\eta_{qr}$ given in the proof of Theorem \ref{WellposedHqr}.
We can chose these constants such that
$$\frac{2\widetilde{K}C_{qr}\delta_{qr}}{1-K_2\eta_{qr}} + \eta_{q,r}<1.$$
As a direct consequence of this and \eqref{limsup} we have that
$$\limsup_{t\to\infty}\norm{\begin{bmatrix}
t^{\frac{\alpha}{2}}u\\ 
t^{\frac{\beta}{2}}\theta 
\end{bmatrix}(\cdot,t)}_{q,\infty,\lambda;r,\infty,\lambda} =0.$$
Our proof is completed.
\end{proof}

\end{document}